\documentclass[12pt]{amsart}
\usepackage[frenchb, english]{babel}
\usepackage{amssymb,amsfonts,amsmath}
\usepackage{mathrsfs}
\usepackage{anysize}
\usepackage{bbm}
\usepackage{url}
\usepackage[leqno]{amsmath}
\usepackage{enumitem}
\usepackage{tikz-cd}
\usepackage{stackengine}
\usetikzlibrary{decorations.pathreplacing}
\usepackage{soul}
\setstcolor{red}

\usepackage{changes}
\definechangesauthor[name=JPG, color=blue]{Ze}
\definechangesauthor[name=BP, color=green]{Benito}

\usetikzlibrary{arrows}
\usetikzlibrary{patterns}
\usepackage{bbding}
\usepackage{datetime}
\usepackage{pgfplots}

\usepackage{caption}
\usepackage{tikz}
\usepackage{accents}
\usetikzlibrary{calc}
\usetikzlibrary{arrows}
\usetikzlibrary{backgrounds}
\usepackage{3dplot}
\tdplotsetmaincoords{50}{135}
\tdplotsetrotatedcoords{0}{0}{0} 

\theoremstyle{plain}
\newcounter{fakecnt}
\newtheorem{theorem}{Theorem}
\newtheorem{lemma}[fakecnt]{Lemma}

\theoremstyle{definition}
\newtheorem{definition}[fakecnt]{Definition}
\usetikzlibrary{decorations.markings}
\tikzset{negated/.style={
        decoration={markings,
            mark= at position 0.5 with {
                \node[transform shape] (tempnode) {${\scriptstyle\setminus} $};
            }
        },
        postaction={decorate}
    }
}

\tikzset{degil/.style={
            decoration={markings,
            mark= at position 0.5 with {
                  \node[transform shape] (tempnode) {$\backslash$};
                  }
              },
              postaction={decorate}
}
}

\title[Chaotic financial time series]{Chaotic time series in financial processes consisting of savings with piecewise constant monthly contributions}

\author{Jos\'e Pedro Gaiv\~ao}
\address{Departamento de Matem\'atica and CEMAPRE/REM, ISEG, Universidade de Lisboa, Rua do Quelhas 6, 1200-781, Lisboa, Portugal.}
\email{jpgaivao@iseg.ulisboa.pt}

\author{Benito Pires}\thanks{Partially supported by grant \#2019/10269-3, S\~ao Paulo Research Foundation (FAPESP)}
\address{Departamento de Computa\c c\~ao e Matem\'atica, Faculdade de Filosofia,
	Ci\^encias e Letras, Universidade de S\~ao Paulo, Ribeir\~ao Preto, SP,
	14040-901, Brazil.}
\email{benito@usp.br}

\date{}


\begin{document}

\maketitle


\marginsize{2.5cm}{2.5cm}{1cm}{2cm}

\begin{abstract} 
We investigate the time series generated by an elementary and deterministic financial process that consists in making monthly contributions to a savings account subjected to the devaluation by a monthly negative real interest rate. The monthly contribution is a piecewise constant function of the account balance. We show that a dichotomy holds for such  a financial time series: either the financial time series are asymptotic to finitely many periodic {sequences} or the financial time series have an uncountable (Cantor) set of $\omega$-limit points.  We also provide  explicit parameters for which the financial process is chaotic in the sense that the financial time series have sensitive dependence on initial conditions at points of a Cantor attractor.   
 \end{abstract}

  

\maketitle

\section{Introduction}\label{in}

The aim of this article is to  build a bridge between the field of Financial Mathematics\linebreak and the research pursued recently on the topological dynamics of piecewise-affine contractions of the interval, e.g.,  [Calderon \textit{et al.}, 2021; Fernandes \& Pires, 2020; Gaiv\~ao \& Nogueira, 2022; Gambaudo \& Tresser, 1988; Laurent \& Nogueira, 2021; Nogueira \& Pires, 2015; Nogueira \textit{et al.}, 2014; Nogueira \textit{et al.}, 2018; Pires, 2019; Pires, 2020].  It turns out that many models in Economics and Finance can be described using very simple piecewise-affine recurrence relations, e.g., [Asano \& Masanori, 2019; Matsuyama \textit{et al.}, 2018; Umezuki \& Yokoo, 2019]),  and some of these relations are defined by piecewise-affine contractions.  Nevertheless,  as we will see in the present work,  even a very simple recurrence relation defined by a piecewise-affine contraction can display nontrivial dynamics {(e.g. limit cycles or Cantor attractors)} and have sensitive dependence on initial conditions,  which is the key ingredient of chaos.  Of course,  in proving the existence of chaos,  the discontinuities of the model will play a fundamental role.


After the discovery of deterministic chaos,  there has been a great deal of effort to apply methods from nonlinear dynamical systems theory to model the source of randomness that is observed in economic and financial time series, e.g., [Day, 1982; Hommes, 2015; Kelsey, 1988].  In {piecewise smooth dynamical systems},  Lyapunov exponents measure the degree of randomness of the system. {While negative Lyapunov exponents indicate the absence of chaos or the presence of low-complexity chaos, po\-sitive Lyapunov exponents is an indication of high-complexity chaos.} There has been an intensive debate in the literature as to whether economic and financial time series exhibit chaos, e.g., [LeBaron, 1994] and several methods and algorithms have been devised to estimate Lyapunov exponents through the analysis of financial data, e.g., [Schittenkopf \textit{et al.}, 2000; Small \& Tse, 2003].


Probably,  the simplest and oldest financial process modelled by an affine recurrence relation is the compound interest.  For instance,  the interest we pay for loans or the interest we receive from deposits.  Traditionally,  the associated interest rate is positive and the future value increases exponentially.  In this article,  we consider instead negative interest rates.  Therefore,  savers will see part of their savings and deposits diminished.  According to the Fisher equation in Financial Mathematics,  a negative real interest rate occurs when inflation is higher than the nominal interest rate, which is currently the situation in Europe, facing a record annual inflation rate of 8.1\% in May 2022, coming from 7.4\% in April 2022, as reported by Eurostat on 17 of June 2022.  Another possibility is that the nominal interest rate drops below 0\%.  This means that banks and other financial institutions have to pay for keeping their excess cash stored in central banks.  
Nowadays,  negative interest rates are more than an intellectual curiosity.  Since the 2008 financial crisis,  most central banks around the world have lowered their interest rates below 0\%.  For instance,  since 2014,  the European Central Bank (ECB) has reduced the deposit facility rate achieving in 2021 a value of -0.5\% (see [Heider \textit{et al.}, 2021]).  Some banks in Europe can no longer absorb the negative interest rates set by ECB and started to charge  households to keep large amounts of money.  Since March 2021,  two of Germany's largest banks,  namely Deutsche Bank and Commerzbank,  charge a 0.5\% fee  for  keeping  large deposits of new costumers. 
Viewed as a central bank monetary policy tool,  negative interest rates work by stimulating spending and punishing savings,  hence boosting the growth of the economy.  Implicitly,  there is the idea that most people are not anymore willing to keep their money in the bank and should prefer to spend or invest it.  However,  a recent experimental study has shown the contrary,  that there is some tolerance to negative interest rates,  meaning that people prefer to save money in the bank, accepting the losses,  rather than spending it or taking further risks (see [Corneille \textit{et al.}, 2021]).

\section{The mathematical model of the financial process}\label{fp} 
 
In the sequel,  we describe in precise terms the model that will be studied. The financial process we consider here is deterministic,  has discrete time $n\in\{1,2,\ldots\}$ and consists in depositing at the end of every month $n$ some amount $d_n$ of money in a savings account with a negative monthly real interest rate $r\in(-1,0)$.  For our purposes, and to simplify matters,  we consider the interest rate constant during the whole time, since already in this simple scenario, as we will see,  the dynamics of the process is highly nontrivial.  Moreover,  we do not include any random terms to model external effects.  In fact,  our goal in this work is to observe some sort of randomness out of very simple and deterministic rules.

We denote by $S_0$ the initial account balance (the principal) and by $S_n$ the account balance at the month $n$, which is completely determined by the initial value $S_0$. We assume that $d_n$ is chosen out of a  set $\{v_1,v_2\}$ of non-negative values according to the following rule
\begin{equation}\label{policy} d_n=\begin{cases} v_1 & \textrm{if} \quad S_{n-1}< \rho\\
v_2 & \textrm{if} \quad S_{n-1}\ge \rho
\end{cases},
\end{equation}
where $\rho>0$ is a constant called \textit{threshold}.
The  sequence $(S_n)_{n\ge 0}$  is called \textit{a financial time series} since it describes the state of the financial process at the time $n$. 

In the framework of dynamical systems, we have that each financial time series $(S_n)_{n\ge 0}$ is a solution of the recurrence relation 
\begin{equation}\label{Sn} S_n = f(S_{n-1}), \quad n=1,2,\ldots,
\end{equation}
where $f:[0,\infty) \to [0,\infty)$ is the piecewise-affine contraction defined by
\begin{equation}\label{pc}
 f(x) = \begin{cases}
(1+r) x + v_1 & \textrm{if}\quad x<\rho\\
(1+r) x + v_2 & \textrm{if}\quad x\ge\rho
\end{cases}.
\end{equation}

The case in which $v_1=v_2=v$ is classical and appears in many textbooks (see [Fulford \textit{et al.} [1997]).  The general term of the time series $(S_n)_{n\ge 0}$ is given explicitly by the closed expression
\begin{equation}\label{closedexpression} 
S_n = (1+r)^n S_0 + \dfrac{(1+r)^n-1}{r} v, \quad n\ge 0.
\end{equation}
In particular, if we choose $S_0=-\dfrac{v}{r}$, then $S_n=-\dfrac{v}{r}$ for all $n\ge 0$. Moreover, for any other value of $S_0$,  
$$\lim_{n\to\infty} S_n =S_\infty:=- \dfrac{v}{r}.$$
Hence, all the time series\footnote{The expression ``time series" has the same plural and singular forms. Here we mean a collection of sequences. Elsewhere, sometimes we use the same expression to mean a unique sequence.} $(S_n)_{n\ge 0}$ are asymptotic to the constant time series $\left(S_\infty \right)_{n\ge 0}$, that is, 
\begin{equation}\label{a1}
 \lim_{n\to\infty} \left\vert S_n -S_\infty\right\vert=0,\quad \forall\, S_0\ge 0.
\end{equation}

To conclude, when $v_1=v_2$, the long-term behaviour of all the financial time series $(S_n)_{n\ge 0}$ is given by the  constant financial time series $\gamma=\left(S_\infty\right)_{n\ge 0}$. In other words, all the financial time series are asymptotic to $\gamma$ in the sense that \eqref{a1} is satisfied.

 In the case in which $v_1\neq v_2$, the financial process does not admit a closed expression like \eqref{closedexpression}. Moreover, in such a case the interval map \eqref{pc} has a discontinuity.   
 For this reason, it is more difficult to analyse its dynamics. {In fact, only recently the topological dynamics of such maps were completely understood (see, for instance, [Calderon \textit{et al.}, 2021; Fernandes \& Pires, 2020; Gaiv\~ao \& Nogueira, 2022; Nogueira \& Pires, 2015; Nogueira \textit{et al.}, 2014; Nogueira \textit{et al.}, 2018; Pires, 2019; Pires, 2020], for very recent results). In the next section, we provide a complete study of such a case.}

\section{Statement of the results}

{Here we consider the financial process described in $\S\ref{fp}$ with  deposit values $v_1$ and $v_2$ different from each other. To encompass the variety of possible behaviours, we need different notions from topological dynamics.}     

 We say that a sequence $(a_n)_{n\ge 0}$ is \textit{periodic} if $a_{i+n}=a_i$ for some $n\ge 1$ and all $i\ge 0$. We call $N=\min\{n\ge 1: a_n=a_0\}$  the \textit{period} of $(a_n)$. We say that a sequence $(b_n)_{n\ge 0}$ is \textit{asymptotic} to a sequence $(a_n)_{n\ge 0}$ if $\lim_{n\to\infty}\left\vert b_n - a_n\right\vert=0$. Since time series are just sequences of numbers, the previous notion translates to  time series. 
 
 \begin{definition}[asymptotically periodic financial process]\label{def1111} We say that the financial process  described in \S\ref{fp} is \textit{asymptotically periodic} if there is a finite collection of {periodic sequences} such that each financial time series of the process is asymptotic to a periodic sequence of the collection.
\end{definition}

{In this way, if $m$ denotes the number of periodic sequences in  Definition \ref{def1111}, then}
 each time series $(S_n)_{n\ge 0}$ of the process, regardless the initial value $S_0$, gets arbitrarily close, as $n$ tends to $\infty$, to a set of  at most $N=N_1+\cdots + N_m$ values, where $N_i$ is the period of the $i$-th periodic {sequence}. The number $N$ can be very large.

Contrasting with the regular and predictable behaviour of  asymptotically periodic financial processes, the financial process considered here may also have a chaotic behaviour. Such examples are rare and difficult to construct because there exist results that rule out such a behaviour for Lebesgue almost every parameter. More precisely, piecewise-affine contractions such as that in \eqref{pc} are typically asymptotically periodic (see, for instance, the main results in Gaiv\~ao \& Nogueira [2022], Nogueira \textit{et al.} [2014], Nogueira \textit{et al} [2018]). In spite of this, in the next section, we provide an existence result for chaotic examples. Now we will explain what we mean by chaotic behaviour. 

There are a variety of definitions of chaos (Devaney chaos, Li-Yorke chaos, {Wiggins chaos}, etc.). 
The most important ingredient of chaos is the notion of sensitive dependence on initial conditions. 
\begin{definition}[sensitive dependence on initial conditions]\label{def0}
We say that the financial time series $(S_n)_{n\ge 0}$ of the financial process described in \S\ref{fp} have \textit{sensitive dependence on  initial conditions at $S_0\ge 0$} if for some positive constant $\eta>0$ the following is true: for each $\epsilon>0$, there exist $S_0'\ge 0$ and $k\in\mathbb{N}$ such that $\left\vert S_0'-S_0\right\vert\le\epsilon$ and
$\left\vert S_k'-S_k\right\vert\ge\eta$, where $(S_n')_{n\ge 0}$ and $(S_n)_{n\ge 0}$ are time series of the  financial process with starting points $S_0'$ and $S_0$,  respectively.
\end{definition}

The \textit{$\omega$-limit set} of a sequence $(S_n)_{n\ge 0}$ is the set
 \begin{equation}\label{limitset} 
\omega((S_n)_{n\ge 0})=\{p\in \mathbb{R}: \exists  n_1<n_2<n_3<\cdots\,\,\textrm{such that}\,\,\lim_{k\to\infty}S_{n_k}=p\}.
 \end{equation}
  The $\omega$-limit sets of the financial process considered here are the collection of limit sets of each one of its time series $(S_n)_{n\ge 0}$. Concerning asymptotically periodic financial processes, such a collection consists of finitely many finite sets (consisting of the terms of each periodic sequence). In general, such $\omega$-limit sets can also be  Cantor sets (i.e. compact, perfect, nowhere dense sets).
  
  In financial processes with sensitive dependence on initial conditions,  it is impossible to predict the value $S_n$ of an individual time series for $n$ large since any small measurement error of $S_0$ will result in a large effect deviating from the true value $S_n$.  Nevertheless,  it is possible to predict the frequency with which a time series $(S_n)_{n\ge 0}$ visits any interval of values $J$. More precisely, we define the frequency of visit of $(S_n)$ to the interval $J$ as the limit (whenever the limit exists):
\begin{equation}\label{freq}
 \textrm{freq}\, \big((S_n)_{n\ge 0}, J \big)=\lim_{N\to\infty}\dfrac{1}{N} \#\left\{0\le n\le N-1: S_n\in J\right\},
\end{equation}
where $\#$ stands for  cardinality.

\begin{definition}[chaotic financial process]\label{cfp} We say that the financial process described in \S\ref{fp} is \textit{chaotic} if there exists a Cantor set $C\subset [0,\infty)$ such that the following statements are true:
\begin{itemize}
\item [$(C1)$] Each financial time series of the  process has the Cantor set $C$ as its $\omega$-limit set;
\item [$(C2)$] The financial time series of the  process have sensitive dependence on initial conditions at each $S_0\in C$; 
\item [$(C3)$] For each interval $J\subset [0,\infty)$, there exists a constant $c_J\ge 0$ such that 
$$ \lim_{N\to\infty}\dfrac{1}{N} \#\left\{0\le n\le N-1: S_n\in J\right\}=c_J
$$
for all financial time series $(S_n)_{n\ge 0}$.
\end{itemize}
\end{definition}

Now that we have introduced and explored the notions of dynamical systems required to understand the long-term behaviour of financial time series, we are ready to state our results.      

\begin{theorem}\label{thm1} Considering the financial process described in $\S\ref{fp}$, one of the following alternatives occurs:
\begin{itemize} 
\item [$(i)$] The financial time series  are asymptotic to at most two periodic {sequences};

\item [$(ii)$] The financial time series have the same Cantor set as their $\omega$-limit set.

\end{itemize} 
\end{theorem}

{In what follows, we denote by {$\boldsymbol{\omega}=\omega_0\omega_1\omega_2\ldots$} the \textit{Fibonacci word}, that is, the sequence of binary digits
\begin{equation}\label{fw}
 \boldsymbol{\omega}=010010100100101001010010010100100101001010010010100 . . .
\end{equation}
defined by {$\omega_i=2+\lfloor (i+1)\varphi\rfloor-\lfloor (i+2)\varphi\rfloor$}, where $\varphi=(1+\sqrt{5})/2$ is the golden ratio and $\lfloor x\rfloor$ denotes the integral part of $x$. Notice that the Fibonacci word is the sequence A003849 in the OEIS\footnote{\textsc{The on line encyclopedia of integer sequences}\textsuperscript{\textregistered}.}.
\begin{theorem}\label{thm2} Let $\boldsymbol{\omega}=\omega_0\omega_1\omega_2\ldots$ be the Fibonacci word given by \eqref{fw},   $1<b<\infty$ and $\delta>0$ be given by
\begin{equation}\label{delta}
\delta=1-{\dfrac{1}{b}} + \dfrac{1}{b}\left(1-\dfrac{1}{b}\right)\sum_{k\ge 0} \omega_{k} b^{-k}.
\end{equation}
Then the financial process described in \S\ref{fp} with the parameters
\begin{equation}\label{rho}
{v_1=1000, \quad v_2=500}, \quad r=\dfrac{1}{b}-1, \quad {\rho=\dfrac{500b}{b-1}\left[ b(1-\delta)+1\right]}
\end{equation}
is chaotic.
\end{theorem}}

{There is nothing special with the Fibonacci word in Theorem \ref{thm2}. In fact, because of Theorem 3.2.11 in Thuswaldner [2020], any Sturmian binary word would work.}


The financial process described in \S\ref{fp} with the parameters provided in Theorem \ref{thm2} is chaotic and, in particular, has sensitive dependence on initial conditions {at points of a Cantor set}. We can observe that phenomenon in a computer simulation of the process. The convergence of the infinite series \eqref{delta} can be very slow if $b$ is close to $1$, or equivalently, if the real interest $r$ is close to $0\%$ in absolute value. To speed up the process and observe the phenomenon of sensitive dependence with a few iterations, we will choose the real interest $r=-50\%$ whose absolute value is high or, equivalenty, we will set $b=2$.

 The figure below  shows the plot of $S_n$ versus $n$ considering the financial process described in \S\ref{fp} with the following parameters
$$v_1=1000, \quad v_2=500, \quad b=2, \quad r=-50\%, \quad {\rho=1709.8034428612914\ldots,} $$
where $\rho$ was calculated using \eqref{delta} and \eqref{rho}.
The plot exhibits the sensitive dependence on initial conditions. \begin{figure}[ht!]
\centering
\includegraphics[scale=1, height=8cm]{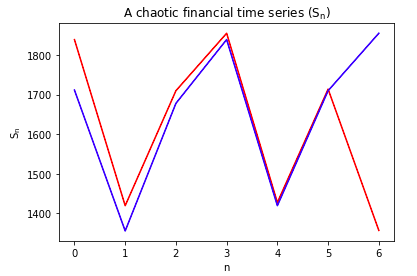}
\end{figure}Computational simulations confirm that  $\textrm{freq}\, \big((S_n)_{n\ge 0}, J \big)$ is constant for $J=[1400,1600]$ (for example) and does not depend on $S_0$. More precisely, we have the following data, considering $N=150$ in \eqref{freq}:

$$ \begin{tabular}{c|c}
$S_0$ &  $\textrm{freq}\, \big((S_n)_{n\ge 0}, [1400,1600] \big)$\\ \hline
$1450$ & 92/150\\
$1380$ & 92/150\\
$1023$ & 92/150\\
$1900$ & 93/150 \\
$800$ & 92/150
\end{tabular}
$$

 \section{Proof of Theorem \ref{thm1}}

We will need some background from Discrete Dynamical Systems. Let $g:I\to I$ be a self-map of an interval $I\subseteq [0,\infty)$. We denote by $g^0$ the identity map on $I$ and by $g^n$ the $n$-ith iterate of $g$ obtained by the composition of $g$ with itself $n$ times. Given $x\in I$, we call the sequence $ O_g(x)=\big(g^n(x)\big)_{n\ge 0}$ 
the \textit{$g$-orbit} of $x$. The $\omega$-limit set of $x$ by $g$ is the set
\begin{equation}\label{glimitset}
\omega_g(x)=\{p\in \mathbb{R}: \exists n_1<n_2<n_3<\cdots\,\, \textrm{such that}\,\, \lim_{k\to\infty} g^{n_k}(x)=p\}.
\end{equation}
A point $x\in I$ is called \textit{periodic} if there exists an integer $n\ge 1$ such that $g^n(x)=x$. In this case,  the $g$-orbit of $x$ is called \textit{periodic}. We say that the $g$-orbit of $x\in I$ is \textit{asymptotic} to the $g$-orbit of $y\in I$ if the sequence $O_g(x)$ is asymptotic to the sequence $O_g(y)$.

In what follows, we consider the financial process described in $\S\ref{fp}$ with parameters
$v_1,v_2\ge 0$, $\rho>0$ and $r\in (-1,0)$. As before, we denote by $S_0$ the initial account balance and by $(S_n)_{n\ge 0}$ a financial time series of the process. By \eqref{Sn} and \eqref{pc}, $(S_n)_{n\ge 0}$ is the $f$-orbit of
$S_0$, where the map $f$ is defined by \eqref{pc}. By \eqref{limitset} and  \eqref{glimitset}, we have that
$ \omega\big((S_n)_{n\ge 0}\big)=\omega_f(S_0)$. Moreover, $(S_n)_{n\ge 0}$ is a periodic time series if and only if the $f$-orbit of $S_0$ is periodic. Yet, $(S_n)_{n\ge 0}$ is asymptotic to a periodic sequence if and only if the $f$-orbit of $S_0$ is asymptotic to a periodic sequence.

\begin{lemma}\label{lem1} Let $S_0\ge 0$. If $\omega_f(S_0)$ is finite, then both the $f$-orbit of $S_0$ and  the financial time series $(S_n)_{n\ge 0}$  are asymptotic to a periodic sequence.
\end{lemma}
\begin{proof} This proof is adapted from Lemma 3.1 in Pires [2019]. 
 We may assume that $$\{S_0, f(S_0), f^2(S_0),\ldots\}$$ is an infinite set, otherwise the $f$-orbit of $S_0$ would be periodic, hence a periodic sequence. Since $\omega_f(S_0)$ is a finite set, we may write $\omega_f(S_0)=\{p_1,\ldots,p_m\}$. Without loss of generality, we may assume that $\omega_f(S_0)\subset (0,\infty)$, thus there exists $\varepsilon>0$ so small that
$$\varepsilon<\frac14 \min_{1\le i\le j\le m} \vert p_i - p_j\vert\quad\textrm{and}\quad
\bigcup_{j=1}^m (p_j-\varepsilon,p_j)\cup (p_j,p_j+\varepsilon)\subset (0,\infty){\setminus}\{\rho\},
$$  
where $\rho\in (0,\infty)$ is the discontinuity of $f$. In particular, if
$$
\mathscr{I}=\left\{(p_1-\varepsilon,p_1), (p_1,p_1+\varepsilon),\ldots,(p_m-\varepsilon,p_m), (p_m,p_m+\varepsilon) \right\},
$$ 
then $f(J)$ is an open interval for every $J\in\mathscr{I}.$

Let $\mathscr{I}'\subseteq \mathscr{I}$ denote the subcollection formed by the intervals that are visited infinitely many times by the  $f$-orbit of $S_0$, that is,
$$
\mathscr{I}'=\left\{J\in\mathscr{I}: \{S_0, f(S_0), f^2(S_0),\ldots\}\cap J\,\,\textrm{is an infinite set}\right\}.
$$
We claim that for each $J_1\in\mathscr{I}'$, there exists $J_2\in\mathscr{I}'$ such that $f(J_1)\subseteq J_2$. Without loss of generality, suppose that $J_1=(p-\varepsilon,p)$, where $p\in \omega_f(S_0)$. As $J_1\subset [0,\infty){\setminus}\{\rho\}$, we have that $f\vert_{J_1}$ is a contraction, thus $f(J_1)$ is an open interval of length smaller than $\varepsilon$. On the other hand, since $J_1\in\mathscr{I}'$, there exists an increasing sequence of integers $0\le k_1<k_2<\cdots$ such that $\left\{f^{k_j}(S_0)\right\}_{j\ge 1}\subset J_1$ and $\lim_{j\to\infty} f^{k_j}(S_0)=p$. Because $f\vert_{J_1}$ is increasing and  continuous, we have that the sequence $\left\{f^{k_j+1}(S_0)\right\}_{j\ge 1}$ is  contained in the open interval $f(J_1)$, has infinitely many distinct terms and  converges to some $q\in \omega_f(S_0)\cap \partial f(J_1)$, where $\partial f (J_1)$ denotes the endpoints of the open interval $f(J_1)$. Putting it all together, we conclude that $f(J_1)$ contains infinitely many points of the $f$-orbit of $S_0$, has length smaller than $\varepsilon$, and has an endpoint in $\omega_f(S_0)$. Therefore, there exists $J_2\in\mathscr{I}'$ such that $f(J_1)\subseteq J_2$.

To finish the proof, notice that there exist $J\in\mathscr{I}'$ and $k'\ge 0$ such that $f^{k'}(S_0)\in J$. By the claim and also because $\mathscr{I}'$ is a finite collection, there exist $1\le i_1<i_2$ and intervals $J_1,\ldots, J_{i_1}, J_{i_1+1},\ldots, J_{i_2}\in\mathscr{I}'$ such that $J_1=J$, $J_{i_1}=J_{i_2}$ and $f(J_i)\subseteq J_{i+1}$ for all $1\le i\le i_2-1$. 
Since $f\vert_{J_i}$ is an increasing contraction for all $1\le i\le i_2-1$, there exists a unique periodic sequence contained in the closure of $\bigcup_{i=i_1}^{i_2-1} J_{i}$ such that the $f$-orbit of $S_0$ is asymptotic to it.
\end{proof}

In order to apply the results in Gambaudo \& Tresser [1988], we need to consider piecewise contractions defined on compact intervals of the form $[0,M]$. The next lemma will be useful to move the problem from $[0,\infty)$ to a compact interval $[0,M]$.

\begin{lemma}\label{lem2} There exists $M>0$ such that $f\big([0,M]\big)\subset [0,M]$. Moreover, for each $S_0\ge 0$, there exist $S_0'\in [0,M]$ and $k\ge 1$ such that  $f^k(S_0)=S_0'$. In particular, 
$\omega_f(S_0)=\omega_f(S_0')$.
\end{lemma}
\begin{proof} Given $j\in \{1,2\}$, let $f_j:[0,\infty)\to [0,\infty)$ be defined by
 $f_j(x)= (1+r)x+v_j$. Let $v_{\max}=\max_j v_j$. Since $v_1,v_2\ge 0$ and $v_1\neq v_2$, we have that $v_{\max}>0$. Set $M=-2v_{\max}/r$. For some $j\in \{1,2\}$, we have that
 $$ 0\le x\le M\implies  0\le f(x) =  f_j(x)  \le (1+r) x + v_{\max}\le  (1+r) M -\dfrac{r}{2}M< M.
 $$
 Hence, $f\big([0,M]\big)\subset [0,M]$. Moreover,
 $$ x\in [0,\infty) \implies 0\le f(x) \le (1+r) x + v_{max}.
 $$
 Hence, proceeding by induction on $k\in\mathbb{N}$, yields
 $$ f^k(x) \le (1+r)^k x + \left[ (1+r)^{k-1} + (1+r)^{k-2}+ \cdots + (1+r) + 1\right]  v_{\max}.   
 $$
In particular, if $k_0=k_0(x)$ is an integer so large that $(1+r)^{k_0}x<M/2$, then
 $$ f^k(x) \le (1+r)^k x + \dfrac{v_{\max}}{-r}\le M,\quad \forall x\ge 0,\quad \forall k\ge k_0.
 $$
 In this way, given $S_0\ge 0$ and $k_0=k_0(S_0)$ as above, then $f^k(S_0)\in [0,M]$ for all
 $k\ge k_0$.  Therefore, if $S_0'=f^{k_0}(S_0)$, then $f^k(S_0')\in [0,M]$ for all $k\ge 0$. In particular, we have that
 $\omega_f(S_0)=\omega_f(S_0')$.
  \end{proof}
  
  \subsection{Proof of Theorem \ref{thm1}} 
 
  \noindent Let $M$ be as in Lemma \ref{lem2},  $f_M:[0,M]\to [0,M]$ be the map defined by $f_M(x)=f(x)$ and $L:[0,1]\to [0,M]$ be the affine map defined by $L(t)=Mt$. By replacing $f_M$ by $L^{-1}f_ML$, we may assume that $M=1$ so that $f_M$ is a piecewise-contraction of the unit interval $[0,1]$ with a single discontinuity point $c\in(0,1)$. By Theorem A in Gambaudo \& Tresser [1988], for every $x\in [0,1]$, we have $$\omega_{f_M}(x) = \omega_{f_M}(f_M(c^-))\cup \omega_{f_M}(f_M(c^+)),$$  where $f_M(c^\pm)=\lim_{x\to c^\pm}f_M(x)$, and the following dichotomy holds: 
\begin{enumerate}
\item $\omega_{f_M}(f_M(c^{-}))$ and $\omega_{f_M}(f_M(c^{+}))$ are finite sets;
\item $\omega_{f_M}(f_M(c^-))=\omega_{f_M}(f_M(c^+))=C$, where $C$ is a Cantor set.
\end{enumerate}  
By Lemma \ref{lem2}, the same conclusion holds considering all $x\in [0,\infty)$ and the piecewise-contraction $f$. By Lemma \ref{lem1}, this can be translated as follows: either $(i)$ the financial time series of the process are asymptotic to at most two periodic sequences or $(ii)$ the financial time series have the Cantor set $C$ as their $\omega$-limit set, which concludes the proof of Theorem \ref{thm1}.
  
 \section{Proof of Theorem \ref{thm2}} 
  
In what follows, we say that a self-map $g$ of $U\subseteq\mathbb{R}$  and a self-map $\widetilde{g}$ of $\widetilde{U}\subseteq\mathbb{R}$  are \textit{topologically semiconjugate}  if there exists a continuous, nondecreasing and surjective map {$h:U\to \widetilde{U}$}, denominated \textit{topological semiconjugacy}, such that $h\circ g=\widetilde{g}\circ h$. If $h$ is also a {homeomorphism}, then we say that $g$ and $\widetilde{g}$ are \textit{topologically conjugate} and we call the map $h$ a \textit{topological conjugacy}. 
  
\begin{lemma}\label{lem3} Let $\alpha=(3-\sqrt{5})/2$, $1\le b<\infty$, {$\boldsymbol{\omega}$} be the Fibonacci word defined by \eqref{fw} and $\delta>0$ be defined by \eqref{delta}.
Then the piecewise contraction $g:[0,1]\to[0,1]$ and the map $T:[0,1]\to [0,1]$ defined by 
$$g(x)=
\begin{cases} \dfrac1bx+\delta & \textrm{if}\quad x\in \big[0,b(1-\delta)\big)\\[0.1in]
 \dfrac1bx+\delta -1 & \textrm{if}\quad x\in \big[b(1-\delta),1\big]
\end{cases},
\quad
T(x)=
\begin{cases}  x+\alpha & \textrm{if}\quad x\in \big[0, 1-\alpha\big)\\
x + \alpha - 1 & \textrm{if}\quad x\in \big[1-\alpha,1\big]
\end{cases}.
$$
satisfy the following statements:
\begin{itemize} \item [$(i)$ ] $g$ and $T$ are topologically semiconjugate via a topological semiconjugacy $h$;
\item [$(ii)$ ] there exists a Cantor set $C\subset [0,1]$ such that $\omega_g(x)= C$ for all $x\in [0,1]$;
\item [$(iii)$ ] if $J\subset [0,1]$ is an interval of positive length intersecting $ {C}$, then ${h}( {C})$ is also an interval with positive length.  
\end{itemize}
 \end{lemma}
 \begin{proof} $(i)$ The proof we present below is an adaptation from the proofs Theorem 2.2 and Corollary 2.5 in Pires [2019]. The approach we use consists of two steps: first we construct  a piecewise contraction $g:[0,1]\to [0,1]$ of the form
 $$g(x)=\begin{cases} \dfrac1bx+b_1 & \textrm{if}\quad x\in \big[0,x_1\big)\\[0.1in]
 \dfrac1bx+ b_2 & \textrm{if}\quad x\in \big[x_1,1\big]
\end{cases}
 $$ 
 topologically semiconjugate to $T$; and then we show that $b_1=\delta$, $b_2=\delta-1$ and $x_1=b(1-\delta)$.    
    
 The sequence $0, \alpha, 2\alpha, 3\alpha-1,\ldots$ whose $k$-th term is $p_k=T^k(0)$, $k\ge 0$, is made up of pairwise distinct terms because, since $\alpha$ is irrational, every $T$-orbit is dense in $[0,1]$. In particular, we have that $\{p_k\}_{k\ge 1}$ is an infinite dense subset of $(0,1)$. Given $k\ge 1$, let 
 $$ \mathcal{L}_k=\{\ell\ge 1: p_{\ell}<p_{k}\},\quad\epsilon_k=\left(1-\dfrac{1}{b}\right) b^{-(k-1)},\quad G_k=\left[ \sum_{\ell\in\mathcal{L}_k} \epsilon_{\ell},\, \epsilon_k + \sum_{\ell\in\mathcal{L}_k} \epsilon_{\ell}
 \right].
 $$
 Notice that $p_k>0$ and $\mathcal{L}_k\neq\emptyset$. Hence, $G_k\subset (0,1)$ is a well-defined interval of length $$\vert G_k\vert=\epsilon_k=\left(1-\dfrac{1}{b}\right)b^{-(k-1)}.$$ We claim that $\{p_k\}_{k\ge 1}$ and $\{G_k\}_{k\ge 1}$ share the same ordering meaning that
 \begin{equation}\label{pk<pj}
 p_k<p_j\iff \sup G_k < \inf G_j.
 \end{equation}
 In fact, $p_k<p_j$ if and only if $ \{k\}\cup \mathcal{L}_k\subset \mathcal{L}_j$, which is equivalent to 
 $$\sup G_k=\epsilon_k+\sum_{\ell\in\mathcal{L}_k} \epsilon_\ell<\sum_{\ell\in\mathcal{L}_j} \epsilon_\ell=\inf G_j.$$
 In particular, we have that the intervals $G_1,G_2,\ldots$ are pairwise disjoint and their union is dense because $\sum_{k=1}^\infty \vert G_k\vert=1$. Applying $(\ref{pk<pj})$ we conclude that if  $J\subset [0,1]$ is an interval and
 $$
 \{m_k\}_{k\ge 1}=\{\ell\ge 1: p_{\ell}\in J\},\quad \textrm{then}\,\, \overline{\cup_{k\ge 1} G_{m_k}}\quad\textrm{is an interval}. 
 $$
 
  Let $\widehat{h}{:}\,\cup_{k\ge1} G_k\to [0,1]$ be the function that on $G_k$ takes the constant value $p_k$. By $(\ref{pk<pj})$, we have that $\widehat{h}$ is nondecreasing and has  dense domain and dense range. Thus, $\widehat{h}$ admits a unique nondecreasing continuous surjective extension $h{:}\,[0,1]\to [0,1]$ to the whole interval $[0,1]$. It is elementary to see that
 $h^{-1}\big(\{p_k\}\big)={G_k}$. 
  
 Let $\widehat{g}:\cup_{k\ge 1} G_k \to \cup_{k\ge 2} G_k$ be such that $\widehat{g}\vert_{G_k}{:}\,G_k\to G_{k+1}$ is an  affine bijection with slope $\dfrac1b$ for every $k\ge 1$. Denote by $I_1, I_2$ the partition of $[0,1]$ defined by
 $I_i=h^{-1}(J_i)$, where $J_1=[0,1-\alpha)$ and $J_2=[1-\alpha, 1]$. Notice that $I_1=[0,x_1)$ and $I_2=[x_1,1]$, where $x_1=h^{-1}(1-\alpha)$. 
  
 We claim that for each $1\le i\le 2$ , there exist a dense subset $\widehat{I}_i$ of $I_i$ and $b_i\in\mathbb{R}$ such that
\begin{equation}\label{lxb}
\widehat{g}(x)=\dfrac{1}{b} x+b_i\quad\textrm{for all}\quad x\in {\widehat I}_i.
\end{equation}
 In order to show that $(\ref{lxb})$ is true, fix $1\le i\le 2$ and let $\{m_k\}_{k\ge 1}=\{\ell\ge 1:p_{\ell}\in J_i\}$,   
 then $\widehat{J_i}={\cup_{k\ge 1} \{p_{m_k}\}}$
  is a dense subset of $ {J_i}$ and  $\widehat{I_i}={\cup_{k\ge 1} G_{m_k}}$ is a dense subset of $I_i$. Moreover, by definition, $\widehat{g}\vert_{G_{m_k}}{:}\,G_{m_k}\to G_{m_k+1}$ is an  affine bijection with slope $\dfrac{1}{b}$ for each $k\ge 1$, which shows that there exists $c_{m_k}\in\mathbb{R}$ such that
  \begin{equation}\label{cmkcmk}
   \widehat{g}(x)=\dfrac{1}{b} x + c_{m_k}\quad\textrm{for all}\quad x\in G_{m_k}.
  \end{equation}
 Let us prove that $\widehat{g}$ is strictly increasing  on $\cup_{k\ge 1} G_{m_k}$. We have that $\widehat{g}$ is strictly increasing on each interval $G_{m_k}$. Let  $y_k<z_j$ be such that $y_k\in G_{m_k}$ and $z_j\in G_{m_j}$, where $k\neq j$ and $\sup G_{m_k}<\inf G_{m_j}$. By $(\ref{pk<pj})$, we have that  
  $p_{m_k}<p_{m_j}$ and $\{p_{m_k},p_{m_j}\}\subset J_i$. Then, since  $T\vert_{J_i}$ is increasing, we have that  $T(p_{m_k})<T(p_{m_j})$, that is,
  $p_{m_k+1} < p_{m_j+1}$. By $(\ref{pk<pj})$ once more, we get $\sup G_{m_k+1}<\inf G_{m_j+1}$.
   By definition, $g(y_k)\in G_{m_k+1}$ and $g(z_j)\in G_{m_j+1}$, thus $g(y_k)<g(z_j)$. This proves that $\widehat{g}$ is increasing on $\cup_{k\ge 1} G_{m_k}$.  It remains to prove that $c_{m_k}$ in $(\ref{cmkcmk})$ is the same for all $k\ge 1$. Let $j\neq k$. Let $x=\sup {G_{m_j}}<\inf G_{m_k}=y$. Notice that because $\widehat{g}$ is increasing on $\cup_{k\ge 1} G_{m_k}$ we have that
   \begin{eqnarray*} 
   \frac1b(y-x)+(c_{m_k}-c_{m_j})&=&
   \widehat{g}(y)-\widehat{g}(x)= \sum_{G_{m_{\ell}}\subset [x,y]} \left\vert \widehat{g}\big(G_{m_{\ell}}\big)\right\vert \\ &=&\frac1b \sum_{G_{m_\ell}\subset [x,y]} |G_{m_{\ell}}|= \frac1b (y-x)
   \end{eqnarray*}
  yielding $c_{m_k}=c_{m_j}$. Thus, $(\ref{lxb})$ is true.
      
  It follows from $(\ref{lxb})$ that $\widehat{g}\vert_{\cup_{k\ge 1} G_{m_k}}$ admits a unique monotone continuous extension to the interval $I_i=h^{-1}(J_i)$. This extension is also an affine map with slope equal to $\frac1b$. Since $i$ is arbitrary, we obtain an injective  piecewise $\frac1b$-affine extension $g$ of $\widehat{g}$  to the whole interval $[0,1]=\cup_{i=1}^2 I_i$. 
        
 We claim that $h\circ g=T\circ h$. In fact, for every $y\in G_k$, we have that 
\begin{equation}\label{eql}
h\big(g(y)\big)=\widehat{h}\big(\widehat{g}(y)\big)=p_{k+1}=T(p_k)=T\big(\widehat{h}(y)\big)=T\big(h(y)\big).
\end{equation}
Hence, $(\ref{eql})$ holds for a dense set of $y\in [0,1]$. By continuity, $(\ref{eql})$ holds for every $y\in [0,1]$. This shows that $g$ is topologically semiconjugate to $T$. Now let us compute the parameters $b_1,b_2$ and $x_1$ in terms of $\delta$. Since $I_1=[0,x_1)$, we have that 
$$ x_1=\left\vert I_1\right\vert=\sum_{p_k\in J_1} \left\vert G_{k}\right\vert=\sum_{T^{k-1}(\alpha)\in J_1} \epsilon_{k}=\sum_{T^{k}(\alpha)\in J_1} \epsilon_{k+1}=
\left(1-\dfrac{1}{b}\right)\sum_{k\ge 0}(2-\theta_k) b^{-k}=b(1-\delta),
$$
 where $\boldsymbol{\theta}=\theta_0\theta_1\ldots$ is the natural $T$-coding of $\alpha$ with respect to the partition
  $J_1=[0,1-\alpha)$, $J_2=[1-\alpha,1]$ of $[0,1]$, that is, $\theta_k=i$ if and only if $T^k(\alpha)\in J_i$.  Since $\alpha=2-\varphi$, where $\varphi=(1+\sqrt{5})/2$ is the golden ratio, we have that
  {$$\boldsymbol{\theta}-1=(\theta_0-1)(\theta_1-1)\ldots$$ 
   is the Fibonacci word $\boldsymbol{\omega}=\omega_0\omega_1\omega_2\ldots$ defined by \eqref{fw}.} {Putting it all together yields 
 $$  \delta=1-\dfrac{1}{b}\left(1-\dfrac{1}{b}\right)\sum_{k\ge 0}(1-\omega_k)b^{-k}=1-\dfrac{1}{b} + \dfrac{1}{b}\left(1-\dfrac{1}{b}\right)\sum_{k\ge 0}\omega_kb^{-k},$$
 showing that $\delta$ is given by \eqref{delta}.}
   
   Since $g$ is topologically semiconjugate to $T$ and $h(x)=0$ iff $x=0$, we have that
 $$ h\big(g(x_1)\big) =T(h(x_1))=T(1-\alpha)=0,\quad \textrm{thus}\quad g(x_1)=0.
 $$
 Therefore, $$0=g(x_1)=g\big(b(1-\delta)\big)=1-\delta+b_2,\quad\textrm{i.e.}\quad b_2={\delta-1.}$$
 Likewise,  since $g$ is topologically semiconjugate to $T$ and $h(x)=1$ iff $x=1$, we have that
 $$ \lim_{\epsilon\to 0+ }h\big(g(x_1-\epsilon)\big)=\lim_{\epsilon\to 0+}T(1-\alpha-\epsilon)=1,\quad\textrm{thus}\quad \lim_{\epsilon\to0+} g(x_1-\epsilon)=1.
 $$
 Therefore,  $$1=\lim_{\epsilon\to0+ }g(x_1-\epsilon)=\lim_{\epsilon\to0+ }g(b(1-\delta)-\epsilon)=1-\delta+b_1,\quad\textrm{i.e.}\quad b_1=\delta.$$
 We have proved that $(i)$ is true.  
   
$(ii)$  It follows from $(i)$ that the maps $g$ and $T$ are topologically semiconjugate via the topological semiconjucagy ${h}$ and that $T$ is minimal, i.e., all $T$-orbits are dense in $[0,1]$. As a result, we have that $g$ is not asymptotically periodic. Since $g$ is a piecewise contraction with one discontinuity, by  proceeding as
   in the proof of Theorem \ref{thm1} we obtain that there exists a Cantor set $C\subset [0,1]$ such that $\omega_g(x)={C}$ for all $x\in [0,1]$.
   
$(iii)$ It follows from the construction of ${h}$ made in the proof of  $(i)$.
 \end{proof}
 \begin{lemma}\label{lem3b} Let $\alpha=(3-\sqrt{5})/2$, $1\le b<\infty$, {$\boldsymbol{\omega}$} be the Fibonacci word defined by \eqref{fw}, $\delta>0$ be defined by \eqref{delta} and $\rho>0$ be defined by \eqref{rho}.
Concerning the piecewise contraction $f:\mathbb{R}\to\mathbb{R}$ and the map $T:[0,1]\to [0,1]$ defined by 
$$f(x)=
\begin{cases} \dfrac1bx+1000 & \textrm{if}\quad x< \rho\\[0.1in]
 \dfrac1bx+500 & \textrm{if}\quad x\ge \rho
\end{cases},
\quad
T(x)=
\begin{cases}  x+\alpha & \textrm{if}\quad x\in \big[0, 1-\alpha\big)\\
x + \alpha - 1 & \textrm{if}\quad x\in \big[1-\alpha,1\big]
\end{cases},
$$
and the interval $$K=\left[ \dfrac{500b}{b-1}(2-\delta), \dfrac{500b}{b-1}\left(3-\delta-\dfrac{1}{b}\right)  \right],$$
the following statements are true:
\begin{itemize} \item [$(I)$ ] $f\vert_{K}$ and $T$ are topologically semiconjugate via a topological semiconjugacy $h$;
\item [$(II)$ ] there exists a Cantor set $ {C}\subset K$ such that $\omega_f(x)= {C}$ for all $x\in\mathbb{R}$;
\item [$(III)$ ] if $J\subseteq K$ is an interval of positive length intersecting $ {C}$, then $h( {C})$ is also an interval with positive length.  
\end{itemize}
 \end{lemma}
 \begin{proof} (I) Let $L:\mathbb{R}\to\mathbb{R}$ be the affine bijection defined by $L(x)=\dfrac{x}{500} +\dfrac{b(\delta-2)}{b-1}$. It is elementary to verify that the map $\overline{g}:\mathbb{R}\to\mathbb{R}$ defined by $\overline{g}=L\circ f\circ L^{-1}$ satisfies
$$\overline{g}(x)=
\begin{cases} \dfrac1bx+\delta & \textrm{if}\quad x <b(1-\delta)\\[0.1in]
 \dfrac1bx+\delta -1 & \textrm{if}\quad x \ge b(1-\delta)
\end{cases}.
$$
and $f$ is topologically conjugate to $\overline{g}$. Moreover, since $L(K)=[0,1]$, we have that

$$ f\vert_K=\overline{g}\vert_{[0,1]}=g,
$$
where $g:[0,1]\to [0,1]$ is as in Lemma \ref{lem3}. Therefore, by Lemma \ref{lem3} $(i)$, the following diagram commutes
\[ \begin{tikzcd}[arrows={-Stealth}]
\mathbb{R}\rar["L"]\dar["f"] & \mathbb{R}\dar["\overline{g}"]   \\%
\mathbb{R}\rar[swap, "{L}"] & \mathbb{R}  
\end{tikzcd}\quad ,\quad
 \begin{tikzcd}[arrows={-Stealth}]
K\rar["L"]\dar["f\vert_K"] & \mbox{[0,1]}\rar["{h_1}"]\dar["g=\overline{g}\vert_{[0,1]}"] & \mbox{[0,1]}\dar["T"] \\%
K\rar[swap, "{L}"] & \mbox{[0,1]}\rar[swap, "{h_1}"] & \mbox{[0,1]}
\end{tikzcd}
\]
 where $h_1$ is the semiconjugacy denoted by $h$ in Lemma \ref{lem3}. Hence, $f\vert_K$ is topologically semiconjugate to $T$ via the topological semiconjugacy $h={h_1}\circ L$, which proves $(I)$. 
 
$(II)$ Let $C_1\subset [0,1]$ be the Cantor set denoted by $C$ in Lemma \ref{lem3}. Let $C\subset K$ be the Cantor set defined by $C=L^{-1}({C_1})$. Since $f$ is topologically
 conjugate to $\overline{g}$ via $L$, it suffices to show that $\omega_{\overline{g}}(x)={C}_1$ for all $x\in\mathbb{R}$. 
 It follows from Lemma \ref{lem3} $(ii)$ that $\omega_{\overline{g}}(x)={C}_1$ for all $x\in [0,1]$ since $\overline{g}\vert_{[0,1]}=g$. It remains to extend such a claim for all $x\not\in [0,1]$. It suffices to show that for each $x\not\in [0,1]$, there exists $k\ge 1$ such $\overline{g}^k(x)\in [0,1]$. Let $x>1>b(1-\delta)$. Then, $\overline{g}^k(x)\ge 0$ for all integers $k\ge 0$ since $\overline{g}$ takes non-negative numbers into non-negative numbers. Hence, it suffices to show that there exists $k\ge 0$ such that $\overline{g}^k(x)\le 1$. Whenever $y>1>b(1-\delta)$, $\overline{g}(y)-y\le\delta-1$, thus $\overline{g}$ moves $y$ to the left by at least $1-\delta$. Hence, $\overline{g}^k(x)\le x - k(1-\delta)$ for all $k\ge 1$ such that $\overline{g}^{k-1}(x)>1$. In this way, for some $k$ large, we have that $\overline{g}^k(x)\le 1$. Now assume that $x<0$, then $\overline{g}^k(x)\le 1$ for all integers $k\ge 0$. Hence, it suffices to show that there exists $k\ge 0$ such that $\overline{g}^k(x)\ge 0$.  Whenever $y<0<b(1-\delta)$, $\overline{g}(y)-y\ge\delta$, thus $\overline{g}$ moves $y$ to the right by at least $\delta$. Hence, $\overline{g}^k(x)\ge x +k\delta$ for all $k\ge 1$ such that $\overline{g}^{k-1}(x)<0$. In this way, for some $k$ large, we have that $\overline{g}^k(x)\ge 0$. Putting it all together, we have that $\omega_{\overline{g}}(x)={C}_1$ for all $x\in\mathbb{R}$. Equivalently, because of the topological conjugacy $L$, we have that $\omega_{f}(x)={C}$ for all $x\in\mathbb{R}$, which concludes the proof of $(II)$. 
 
 $(III)$ Let $J\subseteq K$ be an interval of positive length intersecting $C$. Then $L(J)\subseteq [0,1]$ is an interval of positive length intersecting ${C}_1$.
 By Lemma \ref{lem3} (iii), $h(J)={h_1}\big(L(J)\big)$ has positive length.
   
 \end{proof}

\begin{proof}[Proof of Theorem \ref{thm2}]    
 Let $v_1=1000$, $v_2=500$, $r=\dfrac{1}{b}-1$ and $\rho$ defined by \eqref{rho}. Then, the map $f$ defined in \eqref{pc} is given by
\begin{equation}\label{pc2}
 f(x) = \begin{cases}
\frac1b x + 1000 & \textrm{if}\quad x<\rho\\
\frac1b x + 500 & \textrm{if}\quad x\ge \rho
\end{cases}.
\end{equation}
By Lemma \ref{lem3b} (II), the Cantor set ${C}$ has the property that $\omega_f(x)={C}$ for all $x\in\mathbb{R}$. In this way, if $(S_n)_{n\ge 0}$ is a time series of the process with the parameters given in the statement of Theorem \ref{thm2}, then $S_n=f(S_{n-1})$ for all $n\ge 1$. As a result, 
  $ \omega\big((S_n)_{n\ge 0}\big)=\omega_f(S_0)=C
 $
for all $S_0\ge 0$. We have proved that Condition ($C_1$) in Definition \ref{cfp} is met.

Now let us verify Condition (C2) in Definition \ref{cfp}. Let $S_0\in C$. Given $\epsilon>0$, let
$J=(S_0-\epsilon, S_0+\epsilon)$. We will verify that Definiton \ref{def0} holds true with $\eta=500b\left(1-\dfrac{1}{b}\right)$ (which equals the length of the gap $K{\setminus} f(K)$, where $K$ is as in Lemma \ref{lem3b}). It suffices to show that the discontinuity $\rho\in f^k(J)$ for some $k\ge 0$. In fact, if $k$ is the least non-negative integer such that $\rho\in f^k(J)$, then $f^k(J)$ is an interval with positive length. Moreover, there exists $S_0', S_0''\in J$ such that
$f^k(S_0')<\rho$ and  $f^k(S_0'')>\rho$. It is elementary to verify that $\left\vert f^{k+1}(S_0')-f^{k+1}(S_0'')\right\vert\ge \eta$. It remains to show that $\rho\in f^k(J)$ for some $k\ge 0$. By way of contradiction, assume that $\rho\not\in f^k(J)$ for all $k\ge 0$. Then, $f^k(J)$ is an open interval for all $k\ge 0$. Moreover, by the arguments used in the proof of Lemma \ref{lem3b}, there exists $k_0\ge 0$ such that $f^k({J})\subseteq K$ for all $k\ge k_0$. Moreover, since $S_0\in {C}\cap J$ and since
$\rho\neq f^k(S_0)$ for all $k\ge k_0$, by using a continuity argument, we may show that 
 $f^k({S_0})\in {C}$ for all $k\ge 0$. By Lemma \ref{lem3b} (III),  $h\big(f^{k_0}(J)\big)$ is an interval with positive length. Since $T$ is equivalent to the irrational rotation by  $\alpha$, we have that $T^{-1}$ is equivalent to the irrational rotation by  $-\alpha$. Hence, $T^{-1}$ is also a minimal interval exchange transformation. This means that there exists $m\ge 0$ such that  $T^{-m}(1-\alpha)\in h\big(f^{k_0}(J)\big)$. Equivalently, $1-\alpha\in T^m\Big(h\big(f^{k_0}(J)\big)\Big)$. This implies that $f^{k_0+m}({J})$ hits the discontinuity, which contradicts the induction hypothesis.

To conclude the proof, we will now prove that the frequency with which each financial time series $(S_n)_{n\ge 0}$ visits a given interval $J\subset [0,\infty)$ is constant and does not depend on $S_0$. First notice that if $J\subset [0,\infty){\setminus} K$, then $\textrm{freq}\,\big((S_n)_{n\ge 0},J\big)=0$ for all $S_0\in [0,\infty)$. Thus, to simplify matters, we may assume that $J\subseteq K$. Let $S_0\ge 0$. Then, by Lemma \ref{lem3b} $(I)$,  for all $S_0\ge 0$,
\begin{eqnarray*}
\lim_{N\to\infty}\dfrac{1}{N} \#\left\{0\le n\le N-1: S_n\in J \right\}&=&\lim_{N\to\infty}\dfrac{1}{N} \#\left\{0\le n\le N-1: T^n\big(h(S_0)\big)\in h(J) \right\}\\&=& \textrm{length}\,\big(h(J)\big).
\end{eqnarray*}  
Notice that we have used the fact that since $T$ is equivalent to an irrational rotation, $T$ is uniquely ergodic, thus the Lebesgue measure $\mu$  is its unique invariant probability measure. As a result, the Birkhoff averages converge everywhere to $\int \chi_J \textrm{d}\mu=\textrm{length}\,\big(h(J)\big)$.
\end{proof}






\begin{thebibliography}{23}

 \bibitem{MR3881196}
   Asano, T. \& Masanori, Y. [2019] ``Chaotic dynamics of a piecewise linear model of credit cycles," \textit{J. Math. Econom.} \textbf{80}, pp. 9--21.
  
  \bibitem{MR4266357}
   Calderon A., Catsigeras, E \& Guiraud P. [2021] ``A spectral decomposition of the attractor of piecewise-contracting
  maps of the interval," \textit{Ergodic Theory Dynam. Systems} \textbf{41}, pp. 1940--1960.
  
  \bibitem{CORNEILLE2021101714}
   Corneille, O., D'Hondt,C., De Winne, R., Efendic, E. \& Todorovic, A. [2021]
 ``What leads people to tolerate negative interest rates on their
  savings?," \textit{Journal of Behavioral and Experimental Economics} \textbf{93}, 101714.
  
  \bibitem{10.2307/1831540}
  Day, R. H. [1982] ``Irregular growth cycles," \textit{The American Economic Review} \textbf{72}, pp. 406--414.

\bibitem{MR4120256}  Fernandes, F. \& Pires, B. [2020] ``A switched server system semiconjugate to a minimal interval
  exchange," \textit {European J. Appl. Math.} \textbf{31}, pp. 682--708.
  
  \bibitem{MR1454124}
   Fulford, G., Forrester, P. \& Jones, A. [1997] ``Modelling with differential and difference equations", 
  \textit {Australian Mathematical Society Lecture Series} \textbf{10} (Cambridge University Press, Cambridge).
  
  \bibitem{Gai}
   Gaiv\~ao, J. P. \&  Nogueira, A. [2022] ``Dynamics of piecewise increasing contractions," 
\textit{Bulletin of the London Mathematical Society} \textbf{54}, pp. 482--500,
  .

\bibitem{JMGCT1988}
  Gambaudo, J-. M. \&  Tresser, C. [1988] ``On the dynamics of quasi-contractions,"
\textit{Bol. Soc. Brasil. Mat.} \textbf{19}, pp. 61--114.

\bibitem{HSS21}
   Heider F., F. Saidi, \& Schepens, G. [2021] ``Banks and negative interest rates,"
 \textit{Annual Review of Financial Economics} \textbf{13}, pp. 201--218.

\bibitem{Hommes15}
    Hommes, C. [2015] \textit{Behavioral Rationality and Heterogeneous Expectations in Complex
  Economic Systems} (Cambridge University Press, New York).

\bibitem{10.2307/2663252}
   Kelsey, D. [1988]
 ``The economics of chaos or the chaos of economics,"
\textit{Oxford Economic Papers} \textbf{40}, pp. 1--31.

\bibitem{MR4251937}
   Laurent, M. \& Nogueira, A. [2021]
 ``Dynamics of 2-interval piecewise affine maps and {H}ecke-{M}ahler
  series,"
\textit{J. Mod. Dyn.} \textbf{17}, pp. 33--63.

\bibitem{10.2307/54216}
   LeBaron, B. [1994] ``Chaos and nonlinear forecastability in economics and finance,"
\textit{Philosophical Transactions: Physical Sciences and Engineering} \textbf{348}, pp. 397--404.

\bibitem{MR3893446}
   Matsuyama, K., Sushko, I. \& Gardini. L. [2018]
 ``A piecewise linear model of credit traps and credit cycles: a
  complete characterization."
 \textit{Decis. Econ. Finance}  \textbf{41}, pp. 119--143.

\bibitem{MR3394114}
   Nogueira, A. \& Pires, B. [2015]
 ``Dynamics of piecewise contractions of the interval,"
\textit{Ergodic Theory Dynam. Systems} \textbf{35}, pp. 2198--2215.

\bibitem{MR3225875}
  Nogueira,A.,  Pires, B. \& Rosales, R. A. [2014]
 ``Asymptotically periodic piecewise contractions of the interval,"
 \textit{Nonlinearity} \textbf{27}, pp. 1603--1610.
 
 \bibitem{MR3820005}
 Nogueira, A., Pires, B. \& Rosales, R. A. [2018] 
``Topological dynamics of piecewise {$\lambda$}-affine maps,"
\textit{Ergodic Theory Dynam. Systems} \textbf{38}, pp. 1876--1893.

\bibitem{MR4030596}
  Pires, B. [2019] ``Symbolic dynamics of piecewise contractions,"
\textit{Nonlinearity} \textbf{32}, pp. 4871--4889.

\bibitem{MR4119676}
 Pires, B. [2020] ``Piecewise contractions and {$b$}-adic expansions,"
\textit{C. R. Math. Acad. Sci. Soc. R. Can.} \textbf{42}, pp. 1--9.

\bibitem{SchittenkopfDorffnerDockner+2000}
  Schittenkopf, C.,  Dorffner, G. \&  Dockner, E. J. [2000]
``On nonlinear, stochastic dynamics in economic and financial time
  series," \textit{Studies in Nonlinear Dynamics \& Econometrics} \textbf{4}.

\bibitem{SmallTse+2003}
  Small, M \& Tse, C. K. [2003]
 ``Determinism in financial time series", \textit{Studies in Nonlinear Dynamics \& Econometrics} \textbf{7}.
 
 \bibitem{MR4200104}
  Thuswaldner, J. M. [2020] ``S-adic sequences: a bridge between dynamics, arithmetic, and geometry," \textit{Lecture Notes in Mathematics} \textbf{2273}, eds. Akiyama, S. \&  Arnoux, P., (Springer, Switzerland), Chapter 3, pp. 97--191.

\bibitem{MR3902776}
   Umezuki, Y. \&  Yokoo, M [2019] ``A simple model of growth cycles with technology choice,"
\textit {J. Econom. Dynam. Control} \textbf{100}, pp. 164--175.






\end{thebibliography}

\end{document}